\documentclass[10pt]{article}

\usepackage{amsmath}
\usepackage{amssymb}
\usepackage{amsfonts}
\usepackage{amsthm}
\usepackage{latexsym}

\usepackage[dvips]{graphicx}  
\usepackage{color}
\usepackage{authblk}

\newtheorem{theorem}{Theorem}
\newtheorem{lem}[theorem]{Lemma} 

\newtheorem{prop}[theorem]{Proposition}

\newtheorem*{thmm}{Theorem}

\newtheorem*{problem}{Problem}

\theoremstyle{definition}
\newtheorem{claim}{Claim}
\theoremstyle{remark}

\newtheorem{cProof}{Proof of Claim}

\newcommand{\N}{\mathrm{N}}

\begin{document}

\title{
\textbf{2-connected equimatchable graphs on surfaces } }

%\author{
%Eduard Eiben and Michal Kotrb\v c\'\i k\\[3mm]
%Department of Computer Science\\ Faculty of Mathematics, Physics
%and Informatics\\ Comenius University\\ 842 48 Bratislava,
%Slovakia\\[3mm] {\tt \{eiben, kotrbcik\}@dcs.fmph.uniba.sk}}

\author{Eduard Eiben\thanks{eiben1@uniba.sk} } 
\author{Michal Kotrb\v{c}\'ik\thanks{kotrbcik@dcs.fmph.uniba.sk}}
\affil{Department of Computer Science\\ Faculty of Mathematics, Physics
and Informatics\\ Comenius University\\ 842 48 Bratislava,
Slovakia}

\maketitle

\begin{abstract}
A graph $G$ is equimatchable if  any matching in $G$ is a subset of a maximum-size matching.
It is known that any $2$-connected equimatchable graph is either bipartite or factor-critical. 
We prove that 
for any vertex $v$ of a $2$-connected factor-critical equimatchable graph $G$ and a minimal matching $M$ that isolates $v$ the graph $G\setminus(M\cup\{ v\})$ is either $K_{2n}$ or $K_{n,n}$ for some $n$.
We use this result to improve the upper bounds on the maximum size of $2$-connected equimatchable factor-critical graphs embeddable in the orientable surface of genus $g$ to 
 $4\sqrt g+17$ if $g\le 2$
 and to
 $12\sqrt g+5$ if $g\ge 3$.
Moreover, for any nonnegative integer $g$ we construct a $2$-connected equimatchable factor-critical graph with genus $g$ and more than $4\sqrt{2g}$ vertices, which 
establishes that the maximum size of such graphs is $\Theta(\sqrt g)$. Similar bounds are obtained also for nonorientable surfaces.
Finally,  for any nonnegative integers $g$, $h$ and $k$  we provide a construction of arbitrarily large $2$-connected equimatchable bipartite graphs with orientable genus $g$, respectively nonorientable genus $h$, and a genus embedding with face-width $k$.
\end{abstract}


\section{Introduction}
A graph $G$ is  called equimatchable if any  maximal matching of $G$ has maximum size. 
Equimatchable graphs constitute a classical topic of matching theory investigated for several decades since 
appearing in \cite{grunbaum}, \cite{lewin}, and \cite{meng}.
In particular, Gr\"unbaum  \cite{grunbaum} asked for a characterisation of all equimatchable graphs.
The first step in this direction was a characterisation of all 
randomly-matchable graphs --
 equimatchable graphs with a perfect matching. By a result of Sumner \cite{sumner}, connected randomly-matchable graphs are exactly the complete graphs $K_{2n}$ and complete bipartite graphs $K_{n,n}$ for $n\ge 1$.
The fundamental work \cite{LPP} provides a structural characterisation of equimatchable graphs without a perfect matching using Gallai-Edmonds decomposition, yielding also a polynomial-time algorithm for recognizing equimatchable graphs. Equimatchable factor-critical graphs with a cut-vertex are investigated in \cite{favaron:1986} where it is proved that they contain exactly one cut-vertex $v$, every component of $G-v$ is either $K_{2n}$ or $K_{n,n}$, and $v$ is adjacent to at least two adjacent vertices of every component of $G-v$. A similar description of $2$-connected equimatchable factor-critical graphs with respect to a $2$-cut $\{u,v\}$ is given as Theorem 2.2 of \cite{favaron:1986}: $G\setminus\{u,v\}$ has exactly two components which differ from a complete or complete bipartite graph by at most one edge or by at most two vertices. 
Furthermore, it is proved in \cite{KPS} that if $G$ is a $3$-connected planar graph, $v$ a vertex of $G$, and $M$ a minimal matching isolating $v$, then $G\setminus(V(M)\cup\{v\})$ is randomly matchable and connected, where a matching $M$ is \emph{isolating} a vertex $v$ if $\{v\}$ is a component of $G\setminus V(M_v)$.
It can be easily seen that every component of $G\setminus V(M_v)$ except $\{v\}$ is randomly matchable for every factor-critical graph $G$ and a minimal matching $M_v$ isolating a vertex $v$ of $G$.
Our main theorem below extends these results by showing that $G\setminus(V(M)\cup\{v\})$ has exactly one component.

%{\bf Theorem.} 
\begin{thmm}
Let $G$ be a 2-connected, factor-critical equimatchable graph. Let $v$ be a vertex of $G$ and $M_v$ a minimal matching isolating $v$. %Let $G'=G\setminus(V(M_v)\cup \{v\})$. 
Then $G\setminus(V(M_v)\cup \{v\})$ is isomorphic with $K_{2n}$ or $K_{n,n}$ for some nonnegative integer $n$.
\end{thmm}

%\noindent Furthermore, we use this theorem for investigating the maximum size of a equimatchable graph embeddable in a fixed surface.

The relationship between embeddings of graphs and matching extensions was extensively studied, see for instance \cite{dean}, \cite{KNMS}, or \cite{LZ:2012}. 
The characterisation of equimatchable graphs in \cite{LPP}
%\cite[Theorem 3]{LPP} 
implies that any $2$-connected equimatchable graph is either bipartite or factor-critical. A bipartite graph cannot be factor-critical, since otherwise it would have an odd number of vertices and removing a vertex from the smaller partite set cannot yield a graph with a perfect matching. Therefore, these two classes are disjoint.
%In particular,  
All $3$-connected planar equimatchable graphs are characterised in \cite{KPS} -- there are 23 such graphs and none of them is bipartite.
Let $G$ be a $3$-connected equimatchable graph with an embedding $\Pi$ in the surface of genus $g$.
In \cite{KP} it is proved that if $G$ is either factor-critical, or bipartite and $\Pi$ has face-width at least 3, then the number of vertices of $G$ is bounded from above by $c\cdot g^{3/2}$ for some constant $c$. The proof  uses the fact that there is no such bipartite graph at all and proceeds to restrict the size of equimatchable factor-critical graphs embeddable in a fixed surface. 
%To this end, the concept of isolating matchings is introduced and used as follows.
%A matching $M_v$ isolates a vertex $v$ of $G$ if $\{v\}$ is a component of $G\setminus V(M)$.  
%It is not difficult to see that if $G$ is factor critical and $M_v$ is minimal, then any component of $G' = G\setminus(V(M_v)\cup \{v\})$ is randomly matchable; we provide a proof in Lemma \ref{lemma:Components}. 
First it is shown that if a
$3$-connected graph has many vertices (a number linear in the genus of the graph), then it   has
a vertex $v$ isolated by a matching $M_v$ of size at most $4$.
%Note that the required number of vertices is linear in the genus of the given graph. 
The proof is finished by showing that $G\setminus(V(M_v)\cup \{v\})$
 has at most $\binom{8}{3}(4g+3)$ components.

To bound the maximum size of  equimatchable factor-critical graphs embeddable in a fixed surface, we employ a slightly different strategy: while we allow larger isolating matchings, we use a more precise description of $G\setminus(V(M_v)\cup \{v\})$
given in our main result, Theorem \ref{thm:main}, which implies that it has at most one component. 
As a complete or complete bipartite graph embeddable in the surface of genus $g$ has at most $O(\sqrt g)$ vertices, it suffices to bound the size of isolating matchings. Note that any vertex of degree $d$ admits an isolating matching of size at most $d$. 
The last ingredient of our proof is Lemma \ref{lem:minimalDegree}
showing that either the total number of vertices of the graph, or the minimum degree, is sufficiently small.

Concerning the methods of the paper, while we repeatedly  use the characterisation of randomly matchable graphs from \cite{sumner}, the Gallai-Edmonds decomposition is not used 
 beyond the fact that every $2$-connected equimatchable graph is either bipartite or factor-critical. The constants in the orientable and the nonorientable case are different, hence we state our results explicitly for both cases. However, most of the proofs are virtually identical and in such cases, we omit the proof of the nonorientable case.

The paper is organized as follows. In Section 2 we briefly collect the necessary terms, definitions, and notation regarding matchings and embeddings. In Section 3 we present a proof of our main result
stating that the graph  $G\setminus(V(M_v)\cup \{v\})$ is connected for any 
 $2$-connected factor-critical equimatchable graph $G$ and a 
minimal matching $M_v$ isolating a vertex $v$.
Section 4 is devoted to lower and upper bounds on the maximum size of an equimatchable graph embeddable in a fixed surface.

\section{Preliminaries}
All graphs considered in this paper are finite, simple, and undirected.
A matching is a set of independent edges, that is, a set of edges with no endvertices in common. A matching $M$ of a graph $G$ is called \emph{perfect} if every vertex of $G$ is incident with an edge of $M$. 
A graph $G$ is \emph{factor-critical} if $G\setminus\{v\}$ has a perfect matching for any vertex $v$ of $G$ and
 \emph{equimatchable} if any its maximal matching is maximum.
A graph is called \emph{randomly matchable} if it is equimatchable and has a perfect matching. By \cite{sumner}, the connected randomly-matchable graphs are exactly the  even complete graphs $K_{2n}$ and complete regular bipartite graphs $K_{n,n}$ for all $n\ge 1$.
For a matching $M$, by $|M|$ we denote the size of $M$, that is, the number of edges of $M$, and 
by $V(M)$ we denote the set of vertices incident with the  edges of $M$. A vertex is called \emph{covered} by a matching $M$ if it is incident with an edge of $M$, or equivalently, if it lies in $V(M)$.
For a vertex $v$, a matching $M$ is called a \emph{matching isolating $v$} if $\{v\}$ is a component of $G\setminus V(M)$.
A matching $M$ isolating a vertex $v$ is called \emph{minimal} if no subset of $M$ isolates $v$.
For a graph $G$ and its vertex $v$, the set difference $G\setminus\{v\}$ is for brevity denoted by $G-v$.
By $G\cup H$ we denote the disjoint union of graphs $G$ and $H$. 
An edge with endvertices $u$ and $v$ is denoted by $uv$. By $\delta(G)$ we denote the minimum degree of $G$
and by $\N(v)$ the set of neighbours of a vertex $v$.
 For a deeper account of matching theory the reader is referred to \cite{LP}.

A \emph{surface} is a connected $2$-dimensional manifold without boundary. 
The sphere with $g$ handles (respectively $h$ crosscaps) attached forms a model for orientable surfaces of genus $g$  (nonorientable surfaces of genus $h$) and is denoted by $S_g$ ($N_h$). 
Indeed, the classification theorem for orientable (nonorientable) surfaces states that for any orientable (nonorientable) surface there is exactly one $g\ge 0$ such that $S$ is homeomorphic with $S_g$ 
(exactly one $h\ge 1$ such that $S$ is homeomorphic with $N_h$), see \cite{GT}. The number $g$ (respectively $h$) is called the \emph{orientable (nonorientable) genus} of the surface.
For instance, $S_0$ is the sphere, $S_1$ is the torus, and $N_1$ is the projective plane. 
The \emph{characteristic} of a surface $S$, denoted by $\chi(S)$, equals $2-2g$ if $S$ is homeomorphic with $S_g$, or $2-h$ if $S$ is homeomorphic with $N_h$.
An \emph{embedding} of a graph $G$ in a surface $S$ is a representation of $G$ on $S$ with the following properties. The vertices of $G$ are represented by distinct points of $S$,
the edges of $G$ are represented by disjoint images of the open unit interval, and any open neighbourhood of the image of a vertex intersects images of all edges incident with that vertex, see \cite{GT} or \cite{white:2001} for more details.
An embedding of a graph in a surface is called \emph{cellular} (or 2-cell) if every face of the embedding is homeomorphic with an open disc; we consider only cellular embeddings.
The Euler-Poincar\'e formula (see \cite{GT} or \cite{white:2001}) states that if a graph $G$ with $p$ vertices and $q$ edges is cellularly embedded in a surface $S$  with $r$ faces,
then $p-q+r = \chi(S)$. The \emph{orientable (nonorientable)} \emph{genus} of a graph $G$ is the minimum orientable (nonorientable) genus of a surface into which $G$ can be cellularly embedded and is denoted by $\gamma(G)$, respectively $\tilde \gamma(G)$.
\emph{Face-width}, sometimes called also \emph{representativity} or \emph{planar-width}, of an embedding $\Pi$ in a surface $S$ is the minimum number of faces of $\Pi$ whose union contains a noncontractible cycle in the surface $S$. Several equivalent definitions and further details about face-width can be found in \cite{MT}.

We now present a well-known upper bound on the number of faces of an embedding of a simple graph.
\begin{prop}
\label{prop:2q3r}
Let $G$ be a simple graph with $q$ edges embedded with $r$ faces. Then $2q\ge 3r$.
\end{prop}
\begin{proof}
As $G$ is simple, any face of the embedding has length at least $3$. The result follows from the fact that the union of face boundaries contains every edge precisely twice.
\end{proof}

We repeatedly use the following result due to Ringel and Youngs and Ringel.
\begin{theorem}[\cite{RY:1968, ringel:1959, ringel:1965, ringel:1965b}] 
\label{thm:genus-complete}
The orientable and nonorientable genera of complete and complete bipartite graphs are given by the following formulae:
\begin{eqnarray*}
\gamma(K_n) &=& \left\lceil \frac{(n-3)(n-4)}{12}\right\rceil, n\ge 3; \qquad 
\tilde\gamma(K_n) = \left\lceil \frac{(n-3)(n-4)}{6}\right\rceil, n\ge 3 {\textrm{ \ and\ }} n\neq 7,\  \tilde\gamma(K_7) = 3; 
\\
\rule{0pt}{8mm}\gamma(K_{m,n}) &=& \left\lceil \frac{(m-2)(n-2)}{4}\right\rceil, m,n\ge 2; 
\qquad \tilde\gamma(K_{n,m}) = \left\lceil \frac{(m-2)(n-2)}{2}\right\rceil, m,n\ge 2.
\end{eqnarray*}
\end{theorem}

For a more detailed treatment of  topological graph theory the reader is referred to \cite{GT} or \cite{white:2001}.

%%%%%%%%%%%%%%%%%%%%%%%%%
\section{Isolating matchings 2-connected equimatchable graphs}

This section is devoted to the proof of our main result stated as Theorem \ref{thm:main}.
We start with two lemmas concerning isolating matchings.

\begin{lem}\label{lem:isolation}
Let $G$ be a factor-critical graph. For every vertex $v$ of $G$ there is a matching $M_v \subseteq E(G)$ 
%which isolates 
isolating $v$ %and 
such that $|M_v|\leq \mathrm{deg}(v)$. 

\end{lem}

\begin{proof}
Since $G$ is factor critical, the graph $G' = G-v$ has a perfect matching $M'$. 
Clearly, every neighbour of $v$ is incident with exactly one edge of the matching $M'$. 
Consider a set $M\subseteq M'$ such that $M$ contains precisely those edges from $M'$ that are incident with at least one neighbour of $v$. Then $M$ is the desired matching $M_v$ containing at most $\mathrm{deg}(v)$ edges and isolating $v$.
\end{proof}

Favaron \cite[Theorem 1.1]{favaron:1986} proved that  any connected factor-critical equimatchable graph $G$ with a cut-vertex contains precisely one cut-vertex $v$ and every component of $G-v$ is either $K_{2n}$ or $K_{n,n}$. For equimatchable factor-critical graphs with a $2$-cut $\{u,v\}$, it is still possible to give a description of the structure of $G' =G\setminus\{u,v\}$, albeit it is more complicated: $G'$ has exactly two components and these components are almost complete or complete bipartite, see \cite[Theorem 2.2]{favaron:1986} for  the precise statement 
and details. 
Removing isolating matchings instead of vertex-cuts allows us to obtain a similar description for graphs with arbitrary connectivity in the lemma below.
The underlying idea of its proof is well known, in particular, it was applied in \cite{KPS} and \cite{KP} to prove more specific variants of the result.

\begin{lem}
\label{lemma:Components}
Let $G$ be a connected factor-critical equimatchable graph and $M$   a minimal matching isolating $v$.
Then every component of 
$G\setminus V(M)$ except $\{v\}$ is isomorphic with either $K_{2n}$ or $K_{n,n}$ for some integer $n$.
\end{lem}

\begin{proof}
Let $G' = G\setminus(V(M)\cup \{v\})$ and denote by
$M'$ any maximal matching of $G'$. Clearly, $M=M' \cup M_v$ is a maximal matching of $G$. The graph $G$ is factor-critical and equimatchable, hence $M$ leaves only the vertex $v$ uncovered and $M'$ must be a perfect matching of $G'$. Since arbitrary maximal matching $M'$ of $G'$ is a perfect matching of $G'$, $G'$ is randomly matchable and by \cite{sumner} all of its components are either complete with even number of vertices or complete regular bipartite.
%$K_{2n}$ or $K_{n,n}$. 
\end{proof}

\noindent Note that since $G$ is factor-critical, there always exists a matching isolating any fixed vertex $v$ of $G$.

We say that a subgraph $H_1$ (such as a vertex, edge, or component) of a graph $G$ is \emph{linked} with other subgraph $H_2$ of same graph $G$ if there are vertices $k_1$ of $H_1$ and $k_2$ of  $H_2$ such that $k_1k_2\in E(G)$. 
We are now ready to prove our main result, which sharpens Lemma \ref{lemma:Components} by showing that $G'$ has only one component and
generalizes \cite[Lemma 1.6]{KPS}, which proves that $G'$ has only one component if $G$ is $3$-connected and planar.

\begin{theorem}\label{thm:main}
Let $G$ be a 2-connected, factor-critical equimatchable graph. Let $v$ be a vertex of $G$ and $M_v$ a minimal matching  isolating $v$. 
Then $G\setminus(V(M_v)\cup \{v\})$ is isomorphic with $K_{2n}$ or $K_{n,n}$ for some nonnegative integer $n$.
\end{theorem}

\begin{proof}
We prove the theorem by a series of claims. Let $G'=G\setminus(V(M_v)\cup \{v\})$.

\begin{claim}\label{claim:connectionBetweenEdgeAndComponents}
 If $xy$ is an arbitrary edge of matching $M_v$, then $x$ and $y$ cannot be linked to different components of $G'$.
\end{claim}
\begin{cProof}
We prove the claim by contradiction. Let $C$ and $D$ be different components of $G'$ and suppose that $x$ is adjacent to a vertex $x'$ of $C$ and $y$ is adjacent to a vertex  $y'$ of $D$. Let $M$ be defined by $M = \left(M_v \setminus \{xy\}\right)\cup \left\{ xx', yy'\right\}$. It is easy to see that $M$ is a matching of $G$.
Furthermore, $C-x'$ and $D-y'$ are components of $G\setminus M$. 
From Lemma~\ref{lemma:Components} follows that $C$ and $D$ have even number of vertices and hence both 
$C-x'$ and $D-y'$ have odd number of vertices. It follows that any maximal matching $M'$ such that $M\subseteq M'$ leaves uncovered at least one vertex of both $C-x'$ and $D-y'$.
This is a contradiction with the fact that $G$ is equimatchable and factor-critical. 
\end{cProof}

\begin{claim}\label{claim:connectionBetweenEdgeAndComponents2}
Let $C$ be a component of $G'$ and $xy$ an edge of matching $M_v$ such that $x$ is linked to some vertex $x'$ from $C$. Then $y$ is linked either to $v$ or to some vertex $y'$ of $C$ such that $y' \neq x'$.
\end{claim}
\begin{cProof}
Suppose that $y$ is linked  neither with $C$ nor with $v$. Let $M$ be  defined by $M = \left(M_v \setminus \{xy\}\right)\cup \{xx'\}$. It is easy to see that $M$ is a matching of $G$.
As all neighbours of $v$ are covered by $M$, any maximal matching $M'$ of $G$ such that $M\subseteq M'$ leaves $v$ uncovered. Since $x$ is linked with $C$, by Claim~\ref{claim:connectionBetweenEdgeAndComponents} $y$ cannot be linked to any other component of $G'$. According to our assumption, $y$ is not linked with $v$ or $C$. Therefore, $M'$ leaves uncovered both $v$ and $y$. This is a contradiction with the fact that $G$ is equimatchable and factor-critical, which completes the proof of the claim.
\end{cProof}

\begin{claim}\label{claim:two-independent-edges}
For any edge $e$ of $M_v$ linked with a component $C$ of $G'$, there are two independent edges joining the endvertices of $e$ with $v$ and $C$, respectively.
\end{claim}
\begin{cProof}
Let $e=xy$ and suppose that $x$ is linked with a vertex $x'$ of $C$. By Claim \ref{claim:connectionBetweenEdgeAndComponents2}, $y$ is linked either with $v$ or with some vertex $y'$ of $C$. If $y$ is linked with $v$, then $xx'$ and $yv$ are the two desired edges and we are done. If $y$ is not linked with $v$, then by the minimality of $M_v$ $v$ is linked with $x$. In this case $xv$ and $yy'$ are the desired edges, which completes the proof. 
\end{cProof}

\begin{claim}\label{claim:connectionBetweenEdgeAndComponents3}
Let $C$ be an arbitrary component of $G'$ and $xy$  an edge of $M_v$ linked with $C$. 
If $G'$ has at least two components, then there are two independent edges joining $x$ and $y$ with $C$.
\end{claim}
\begin{cProof}
Without loss of generality assume that $x$ is adjacent to a vertex $x'$ of $C$ and suppose to the contrary that $y$ is not adjacent to a vertex of $C$ different from $x'$. Let $D$ be a component of $G'$ different from $C$.
Since $G$ is $2$-connected, $D$ is linked with at least two vertices of $G\setminus V(D)$. Furthermore, the fact that $v$ is not linked with $D$ implies that these two vertices must be vertices of $M_v$. Because $x$ is linked with $C$, from Claim \ref{claim:connectionBetweenEdgeAndComponents} we get that $y$ cannot be linked with $D$ and thus at least one of the vertices of $M_v$ linked with $D$ is different from both $x$ and $y$. Let $x_1y_1$ be an edge of $M_v$ linked with $D$ such that $x_1y_1\neq xy$.
According to Claim \ref{claim:two-independent-edges} we can assume that $x_1$ is adjacent to a vertex $x_1'$ 
from $D$ and $y_1$ is adjacent to $v$. It is clear that the set $M$ defined by $M= (M_v\setminus\{xy,x_1y_1 \})\cup \{xx',x_1x_1', y_1v\}$ is a matching of $G$. 
Claim \ref{claim:connectionBetweenEdgeAndComponents} implies that $y$ is not linked with any component of $G'$ different from $C$ and in particular, it is not linked with $D$.
According to our assumption, $y$ is not adjacent to any vertex of $C-x'$. It follows that any maximal matching $M'$ such that $M\subseteq M'$ leaves uncovered $y$ and one vertex of both $C$ and $D$. This contradicts equimatchability and factor-criticality of $G$ and completes the proof of the claim.
\end{cProof}

\begin{claim}\label{claim:connectionBetweenEdges}
Let $e$ and $f$ 
be two edges of $M_v$ 
linked with two different components of $G'$.
Then $e$ and $f$ are not linked.
\end{claim}
\begin{cProof}
Let $e=x_1y_1$ and $f=x_2y_2$. Assume that $e$ is linked with a component $C$ of $G'$ and $f$ is linked with a component $D$ of $G'$. Claim \ref{claim:connectionBetweenEdgeAndComponents3} implies that both $x_1$ and $y_1$ are linked with $C$ and both $x_2$ and $y_2$ are linked with $D$. Suppose to the contrary that that $e$ and $f$ are linked; we can assume that they are linked by edge $x_1x_2$. Let $y_1'$ be a vertex of $C$ adjacent to $y_1$ and $y_2'$ a vertex of $D$ adjacent to $y_2$. Clearly, the set $M$ defined by 
$M = \left(M_v \setminus \left\{x_1y_1,x_2y_2\right\}\right) \cup \left\{x_1x_2,y_1y_1',y_2y_2'\right\}$ is a matching of $G$ and any maximal matching $M'$ such that $M\subseteq M'$ leaves unmatched $v$ and at least one vertex of both $C$ and $D$, again contradicting the equimatchability and factor-criticality of $G$.
\end{cProof}

\begin{claim}\label{claim:connectionBetweenEdges4}
Let $e$, $f_1$, and $f_2$ be edges of $M_v$ and $C$ and $D$ two different components of $G'$ such that $C$ is linked with $f_1$ and $D$ is linked with $f_2$.
If $e$ is not linked with $C$, then it is not linked with $f_1$.
\end{claim}
\begin{cProof}
Let $e=uw$, $f_1=x_1y_1$, and $f_2=x_2y_2$, and for the contrary suppose that $e$ is linked with $f_1$.
Since $e$ and $f_1$ are linked, by Claim~\ref{claim:connectionBetweenEdges}  $e$ is not linked to any component of $G'$ different from $C$. Moreover, by our assumption $e$ is not linked with $C$.
% and, by Claim~\ref{claim:connectionBetweenEdges}, it is not linked to any component of $G'$ different from $C$.
By  Claim \ref{claim:connectionBetweenEdgeAndComponents3} there are two independent edges joining $f_1$ and $C$ and two independent edges joining $f_2$ and $D$. Therefore, we can assume that $u$ is linked with $x_1$.
As $M_v$ is minimal, $f_2$ is linked with $v$; let $x_2$ be adjacent to $v$. 
Let $y_1'$ be a vertex of $C$ adjacent to $y_1$ and $y_2'$ a vertex of $D$ adjacent to $y_2$.
It is clear that the set $M$ defined by $M=\left(M_v \setminus \left\{ 
e,f_1,f_2\right\}\right)\cup \left\{ux_1,y_1y_1',vx_2,y_2y_2'\right\}$ is a matching of $G$. 
%It is easy to see that
Since $e$ is not linked with any component of $G'$,
 any maximal matching $M'$ of $G$ such that $M\subseteq M'$ leaves unmatched the vertex $w$ and one vertex of both $C$ and $D$, which contradicts the fact that $G$ is equimatchable and factor-critical.
\end{cProof}

\begin{claim}\label{claim:cutvertex}
If $G'$ has at least two components, then $v$ is a cutvertex.
\end{claim}
\begin{cProof}
Our aim is to show that in $G-v$ there is no path between arbitrary two components of $G'$. 
We proceed by contradiction: suppose there is such a path and among all such paths, choose a path that minimizes the number $k$ of edges of $M_v$ incident with it. Denote one of the paths with $k$ minimal by  $P$ and by $C$ and $D$ the components of $G'$ joined by $P$.
From the fact that $C$ and $D$ are components of $G'$ follows that they cannot be linked directly, and consequently $k>0$. 
Let $e$ and $f$ be the first, respectively the last, edge of $M_v$ incident with $P$.
As no other component of $G'$ is linked with either $C$ or $D$, we get that $e$ is linked with $C$ and $f$ is linked with $D$. 
From Claim \ref{claim:connectionBetweenEdgeAndComponents3} follows that both endvertices of $e$ are linked with $C$ and then Claim \ref{claim:connectionBetweenEdgeAndComponents} implies that $e$ is not linked with $D$. Therefore, $e$ and $f$ are distinct and  $k>1$. Notice that $k=2$ is equivalent with $e$ and $f$ being linked, which is not possible due to Claim \ref{claim:connectionBetweenEdges}. 
Suppose that $k\ge 3$. By the minimality of $k$, there is an edge $a$ of $M_v$ such that $a$ is linked with $e$, but not with $C$. However, this contradicts Claim \ref{claim:connectionBetweenEdges4} and hence $k\ge 3$ is not possible.
We conclude that any  path between $C$ and $D$ contains $v$. Since $G$ is connected, there is at least one such path. Consequently, $v$ is a cutvertex of $G$, which completes the proof of the claim.
\end{cProof}

From the fact that $G$ is $2$-connected and from Claim~\ref{claim:cutvertex} it follows that $G'$ has only one component. Lemma~\ref{lemma:Components} implies that this component is either $K_{2n}$ or $K_{n,n}$, which completes the proof.
\end{proof}

The characterisation of equimatchable factor-critical graphs with a cut-vertex in \cite{favaron:1986} implies that in such graphs $G\setminus V(M_v)$ can have arbitrarily-many components and therefore, Theorem \ref{thm:main} cannot be extended to graph that are not $2$-connected.

%%%%%%%%%%%%%%%%%%%%%%%%%%%%%%%%%%%%%%%%%%%%%%%%%%
\section{Size of $2$-connected equimatchable graphs on surfaces}
\label{section:size}
The aim of this section is to obtain good lower and upper bounds on the maximum size of equimatchable factor-critical graphs  embeddable in the surface of arbitrary fixed genus using Theorem \ref{thm:main}.
We start by showing that there are arbitrarily large equimatchable factor-critical graphs with a cutvertex and any given genus.

\begin{prop}
For any nonnegative integers $g$, $h$, and $k$ there exist connected factor-critical equimatchable graphs $G$ and $\tilde G$ with at least $k$ vertices such that $G$ has  orientable genus $g$ and $\tilde G$ has nonorientable genus $h$.
\end{prop}
\begin{proof}
Let $n$ be an integer such that $K_{2n+1}$ has orientable genus $g$ and  $v$ an arbitrary vertex of $K_{2n+1}$. 
Take $k$ copies of the triangle $K_3$ and designate one vertex in each copy.
It is easy to verify that the graph obtained by vertex amalgamation of $K_{2n+1}$ at $v$ and $k$ triangles at the designated vertices is a connected factor-critical equimatchable graph with genus $g$ and at least $k$ vertices. The proof of the nonorientable case is analogous.
\end{proof}

It is easy to see and well-known that any complete bipartite graph $K_{m,n}$  is equimatchable.
% and any its maximum matching has size $n$.

\begin{prop}
\label{prop:bip-equi}
For any integers $m$ and $n$ such that $m\ge n$ the complete bipartite graph $K_{m,n}$ is equimachable and its maximum matching has size $n$.
\hfill $\square$
\end{prop}

The next three results yield a construction of large $2$-connected equimatchable factor-critical graphs embeddable in any fixed surface.

\begin{lem}\label{lem:union}
Let $u$ and $v$ be adjacent vertices of $K_{n,n}$ and  $x$ and $y$ different vertices from the larger partite set of $K_{m+1,m}$ for some $m$ and $n$. Then the graph $G$ defined by $G = K_{n,n} \cup K_{m+1,m} \cup \{ux, vy \}$ is factor-critical and equimatchable.
\end{lem}

\begin{proof}
Denote by $H_1$ the copy of $K_{n,n}$ and by $H_2$ the copy of $K_{m+1,m}$ in $G$, thus $G = H_1 \cup H_2 \cup \{ux,vy\}$.
First, we show that $G$ is factor-critical, that is, the graph $G-w$ has a perfect matching for any vertex $w$ of $G$. Denote by $A$ and $B$ the larger, respectively the smaller partite set of $H_2$. We distinguish three cases.

\textit{Case 1: $w$ is a vertex of $H_1$.} We can assume that $w$ is in same partite set as $v$. Clearly, there is a perfect matching $M_1$ of $H_1 \setminus \{u, w\}$ and a perfect matching $M_2$ of $H_2 -x$. It follows that matching $M$ defined by $M = M_1 \cup M_2 \cup \{ux\}$ is a perfect matching of $G -w$.

\textit{Case 2: $w$ is a vertex of $A$.} Take any perfect matching $M_1$ of $H_1$ and any perfect matching $M_2$ of $H_2- w$. The matching $M$ defined by $M = M_1 \cup M_2$  is clearly a perfect matching of $G -w$.

\textit{Case 3: $w$ is a vertex of $B$.} Take any perfect matching  $M_1$ of $H_1 \setminus \{u, v\}$ and any perfect matching $M_2$ of $H_2 \setminus \{w,x, y\}$. It is easy to see that the matching $M$ defined by $M = M_1 \cup M_2 \cup \{ux, vy\}$  is a perfect matching of $G -w$.

Now we show that $G$ is equimatchable by proving that any matching $M$ of $G$ is a subset of a maximum matching.
As $G$ is factor-critical, any maximum matching of the graph $G$ leaves precisely one vertex uncovered. 
We distinguish three cases according to which of the edges $ux$ and $vy$ lie in $M$.

\emph{Case 1: neither $ux$ nor $vy$ is an edge of $M$.} Clearly, $M$ is a disjoint union of matchings $M_1$ of $H_1$ and $M_2$ of $H_2$. Since both $H_1$ and $H_2$ are equimatchable by Proposition \ref{prop:bip-equi}, the matchings $M_1$ and $M_2$ can be extended to maximum matchings $M_1'$ of $H_1$ and $M_2'$ of $H_2$, respectively. Clearly, the matching $M_1'$ covers all vertices of $H_1$ and $M_2'$ covers all but one vertices of $H_2$. Therefore, the matching $M'$ defined by $M'=M_1' \cup M_2'$ is a maximum matching of $G$.

\emph{Case 2: either $ux$ or $vy$ is an edge of $M$, but not both.} We can assume that $ux$ is an edge of $M$ and $vy$ is not an edge of $M$. Let $H_1' = H_1 - u$ and $H_2'=H_2-x$. Observe that $H_1'$ is isomorphic with $K_{n,n-1}$ and $H_2'$ is isomorphic with $K_{m,m}$. Consequently, by Proposition \ref{prop:bip-equi} $H_1'$ is equimatchable and any its maximum matching misses exactly one vertex and $H_2'$ is 
equimatchable and has a perfect matching. The matching $M$ is a disjoint union of matching $M_1$ of $H_1'$, matching $M_2$ of $H_2'$, and the edge $ux$.
By equimatchability of $H_1'$ and $H_2'$ the matching $M_1$ extends to a matching $M_1'$ of $H_1'$ missing exactly one vertex and $M_2$ extends to a perfect matching of $H_2'$. The matching $M'$ defined by $M' = M_1'\cup M_2' \cup \{ux\}$ is the desired matching of $G$ missing exactly one vertex.

\emph{Case 3: both $ux$ and $vy$ are edges of $M$.} Let $H_1' = H_1 \setminus \{u,v\}$ and $H_2' = H_2 \setminus \{x,y \}$. Observe that $H_1'$ is isomorphic with $K_{n-1,n-1}$ and $H_2'$ is isomorphic with $K_{m-1,m}$ and thus, by Proposition \ref{prop:bip-equi}, both are equimatchable, $H_1'$ admitting a perfect matching and $H_2'$ a matching missing exactly one vertex.
Clearly, $M$ is a disjoint union of matchings $M_1$ of $H_1'$, $M_2$ of $H_2'$, and edges $ux$ and $vy$. Again, $M_1$ extends to a perfect matching $M_1'$ of $H_1'$ and $M_2$ extends to a matching $M_2'$ of $H_2'$ missing exactly one vertex. Therefore, the matching $M'$ defined by $M' = M_1' \cup M_2' \cup \{ux,vy\}$ is a matching of $G$ missing exactly one vertex.
\end{proof}

Although we need the following lemma only for $K_{n,n}$ and $K_{m+1,m}$, we state it in a general form since the proof is identical.

\begin{lem}
\label{lemma:genus-bipartite-join}
Let $a,b,c,d$ be positive integers such that $c>d$.
Let $u$ and $v$ be two adjacent vertices of $K_{a,b}$. 
Then there are two distinct vertices  $x$ and $y$ from the larger partite set of $K_{c,d}$ such that the graph $G$ defined by $G = K_{a,b} \cup K_{c,d} \cup \{ux, vy \}$  has the genus equal to  $\gamma(K_{a,b}) + \gamma(K_{c,d})$. Similarly,
 there are two distinct vertices   $\tilde x$ and $\tilde y$ from the larger partite set of $K_{c,d}$ such that the graph $G$ defined by $\tilde G = K_{a,b} \cup K_{c,d} \cup \{u\tilde x, v\tilde y \}$  has the genus equal to  $\tilde\gamma(K_{a,b}) + \tilde\gamma(K_{c,d})$.

\end{lem}
\begin{proof}
We start by constructing the desired graph $G$ and its embedding of genus $\gamma(K_{a,b}) + \gamma(K_{c,d})$. Denote by $H_1$ a copy of $K_{a,b}$ and by $H_2$ a copy of $K_{c,d}$.
Let $\Pi_i$ be a minimum-genus embedding of $H_i$ for $i\in\{1,2\}$. Since the vertices $u$ and $v$ are adjacent, there is a face $F_1$ of $\Pi_1$ such that both $u$ and $v$ lie on the boundary of $F_1$. Because $H_2$ is bipartite, any face of $\Pi_2$ has length at least four and thus contains at least two vertices from the larger partite set of $H_2$. Let $x$ and $y$ be arbitrary two vertices of the larger partite set of $H_2$ that lie together on the boundary of a face $F_2$ of $\Pi_2$ and let $G = H_1 \cup H_2 \cup \{ux, vy \}$. 
 Adding one end of the edge $ux$ into the interior of $F_1$ and the other end of $ux$ into the interior of $F_2$ 
merges these faces into one face $F$, producing 
an embedding $\Pi$ of connected graph $H_1\cup H_2 \cup\{ux\}$ in the surface of genus $\gamma(H_1) + \gamma(H_2)$.
Consequently, both $v$ and $y$ lie on the boundary of $F$ and the edge $vy$ can be added into $\Pi$ without raising the genus, yielding the desired embedding of $G$ in the surface of genus  $\gamma(H_1) + \gamma(H_2)$. Since $H_1$ and $H_2$ are disjoint subgraphs of $G$, we get that $\gamma(G)\ge \gamma(H_1) + \gamma(H_2)$, which completes the proof of the orientable case. The proof of the nonorientable case is completely analogous.
\end{proof}

\begin{theorem}\label{thm:lowerBound}
For any nonnegative integers $g$ and $h$ there exist  $2$-connected factor-critical equimatchable graphs $G$ and $\tilde G$ such that $G$ has  orientable genus $g$ 
and at least $4\lfloor \sqrt{2g}\rfloor + 5$ vertices and $\tilde G$ has nonorientable genus $h$ and at least $4\lfloor \sqrt h \rfloor + 5$ vertices.
\end{theorem}
\begin{proof}
Let $n$ and $m$ be maximum integers such that  $K_{n,n}$ is embeddable in the orientable surface of genus  $\lfloor g/2 \rfloor$ and $K_{m+1,m}$ is embeddable in the orientable surface of  genus $\lceil g/2 \rceil$.
Let $u$ and $v$ be two adjacent vertices of $K_{n,n}$. By Lemma \ref{lemma:genus-bipartite-join} there are two vertices $x$ and $y$ of $K_{m+1,m}$ such that the graph $G$ defined by $G=K_{n,n} \cup K_{m+1,m} \cup \{ux, vy \}$ is $2$-connected with orientable genus $\gamma(K_{n,n}) + \gamma(K_{m+1,m}) = \lfloor g/2 \rfloor + \lceil g/2 \rceil = g$. 
By Lemma \ref{lem:union}, the graph $G$ is equimatchable and factor-critical.

To complete the proof it suffices to bound the number of vertices of $G$ from below by calculating the value of $n$ and $m$.
First suppose that $g$ is even. It is not difficult to verify that $n=\lfloor \sqrt{2g} \rfloor + 2$ and that $m = \lfloor (3+\sqrt{8g+1})/2 \rfloor$. 
Since $\lfloor 2\alpha \rfloor\ge  2\lfloor \alpha \rfloor \ge \lfloor 2\alpha \rfloor -1 $ holds for any positive real number $\alpha$, we get that $K_{m+1,m}$ has  $2m+1 \ge 3+\lfloor \sqrt{8g+1}\rfloor \ge 3+ 2\lfloor\sqrt{2g}\rfloor$ vertices. Consequently, $G$ has at least $4\lfloor \sqrt{2g}\rfloor + 7$ vertices.
If $g$ is odd, then $n=\lfloor \sqrt{2g - 2}\rfloor +2$ and $m=\lfloor (3+\sqrt{8g + 9})/2\rfloor$. Since $\lfloor \sqrt{2g - 2}\rfloor \ge \lfloor \sqrt{2g}\rfloor - 1$ for any positive integer $g$, 
$K_{n,n}$ has   $2(2 +\lfloor \sqrt{2g - 2}\rfloor) \ge 2 + 2\lfloor \sqrt{2g}\rfloor$ vertices. Similarly as in the case of even $g$ we get that $K_{m+1,m}$ has at least $3+ 2\lfloor\sqrt{2g}\rfloor$ vertices. Therefore, $G$ has at least $4\lfloor \sqrt{2g}\rfloor + 5$ vertices, which completes the proof of the orientable case. The nonorientable case is analogous.
\end{proof}

The following four lemmas  enable us to obtain upper bounds on the size of $2$-connected equimatchable factor-critical graphs embeddable in a fixed surface.

\begin{lem}\label{lem:maxComponent}
If $G$ is a 
randomly matchable
graph 
embeddable in the orientable surface of genus $g$ (nonorientable genus $h$), then $|V(G)| \le 4 + 4\sqrt{g}$, respectively $|V(G)|\le 4+2\sqrt{2h}$.
\end{lem}
\begin{proof}
If $G$ is a complete graph embeddable in the orientable surface of genus $g$, then $|V(G)|\leq (7+\sqrt{1+48g})/2$ by Theorem \ref{thm:genus-complete}. If $G$ is a complete regular bipartite embeddable in the orientable surface of genus $g$, then $|V(G)| \leq 4+4\sqrt{g}$ by Theorem \ref{thm:genus-complete}.
The 
inequality $(7+\sqrt{1+48g})/2 \leq 4+4\sqrt{g}$,  
which holds for any $g\ge 0$, implies the result in the orientable case. The proof of the nonorientable case is analogous.
\end{proof}

\begin{lem}\label{lem:minimalDegree}
If $G$ has 
a cellular embedding in a surface $S$ and
more than
$$
\frac{6\chi(S)}{5-d}
$$
vertices for some $d\ge 6$, then $\delta(G)\le d$.
\end{lem}
\begin{proof}
We prove the lemma by contradiction. 
Suppose that $\delta(G) \ge d+1$ and consider an embedding of $G$ in the surface $S$. Denote by $p$, $q$, and $r$ the number of vertices and  edges of $G$ and the number of faces of the embedding, respectively. 
As $\delta(G) \geq d+1$ we have $2q\geq (d+1) p$.
Since $G$ is a simple graph,  $2q\ge 3r$ holds by Proposition~\ref{prop:2q3r}.
Substituting the expressions for $p$ and $r$ from the  last two inequalities into 
Euler-Poincar\'e formula  yields
$$
\chi(S) = p - q + r \leq \frac{2q}{d+1} - q + \frac{2q}{3} = \frac{q(5-d)}{3(d+1)}. 
$$
Using $d\geq 6$ and $2q\geq (d+1) p$ we have
$$
\frac{q(5-d)}{3(d+1)} \leq \frac{p(d+1)}{2}\cdot \frac{5-d}{3(d+1)}
$$
and therefore
$$
\chi(S) \leq \frac{p(5-d)}{6},
$$
which contradicts the assumption of the lemma.
\end{proof}

\begin{lem}\label{lem:maxVerticesByDegree}
Let $G$ be a 2-connected, factor-critical equimatchable graph embeddable in the  surface with orientable  genus $g$, respectively nonorientable genus $h$.
If $G$ has a vertex of degree at most $d$, 
then $|V(G)|\leq 5+2d+4\sqrt{g}$, respectively $|V(G)|\leq 5+2d+2\sqrt{2h}$.
\end{lem}
\begin{proof}
Let $v$ be a vertex of $G$ with degree $d$ in $G$ and
 $M_v$ a minimal matching that isolates $v$. By Lemma~\ref{lem:isolation} $M_v$ covers at most $2d$ vertices. 
Let $G'=G\setminus(V(M_v)\cup \{v\})$. By Theorem~\ref{thm:main} $G'$ has at most one component, this component is randomly matchable,  
and Lemma~\ref{lem:maxComponent} yields that $|V(G')| \leq 4+4\sqrt{g}$, respectively $|V(G')| \le 4 +2\sqrt{2h}$. Hence $G$ is a union of vertex $v$, matching $M_v$, and $G'$, and in the orientable case we have 
$$|V(G)|
=|\{v\}| + |V(M_v)| + |V(G')| \le 1 + 2d + |V(G')| 
\leq 1+2d+4+4\sqrt{g} \le 5+2d+4\sqrt{g}.$$
In the nonorientable case $|V(G)|
\le 1 + 2d + |V(G')|  \le 5+2d+2\sqrt{2h}$, which completes the proof.
\end{proof}

\begin{lem}\label{lem:upperbound}
%If for some $d_0\geq 6$ holds inequality
For any $d\ge 6$ such that
$$\frac{6\left(2-2g\right)}{5-d} \leq 5+2d+4\sqrt{g}, 
\quad \textrm{respectively} \quad 
\frac{6\left(2-h\right)}{5-d_0} \leq 5+2d_0+2\sqrt{2h},$$ 
%every 
the maximum size of a $2$-connected factor-critical equimatchable graph embeddable in the surface with orientable
genus $G$ (nonorientable genus $h$) 
is at  most $5+2d+4\sqrt{g}$, respectively $5+2d+2\sqrt{2h}$ vertices. 
\end{lem}

\begin{proof}
We prove the lemma by contradiction.
    Let $d$ be an integer such that $d\ge 6$ and let $G$ be a $2$-connected factor-critical equimatchable graph
embeddable in the orientable surface of genus $g$ with  $|V(G)|>5+2d+4\sqrt{g}$. 
    By our assumption $$|V(G)|>\frac{6(2-2g)}{5-d}$$ 
	and thus
    by Lemma~\ref{lem:minimalDegree} $G$ has a vertex with degree $d'$ such that $d'\leq d$.
    Consequently, by Lemma~\ref{lem:maxVerticesByDegree} $G$ has at most 
    $5+2d'+4\sqrt{g} \leq 5+2d+4\sqrt{g}$ vertices, which is a contradiction. The nonorientable case is analogous.
\end{proof}

\begin{theorem}\label{theorem:mainResult}
Let $m(g)$, respectively $\tilde m(h)$, denote the maximum number of vertices of a $2$-connected factor-critical equimatchable graph embeddable in the orientable surface of  genus $g$, respectively nonorientable surface of genus $h$. Then the following inequalities hold.\\
\textit{i)} 
If $g\le 2$ and $h\le 2$, then 
$$
4\sqrt{2g} + 1 \le m(g) \leq 4\sqrt{g} + 17 
\quad \textrm{and} \quad
4 \sqrt{h} +1 \le \tilde m(h) \le 2\sqrt{2h} +17.
$$
\textit{ii)} If $g\geq 3$ and $h\geq 3$, then 
$$
4\sqrt{2g} + 1 \leq m(g) \leq c_g\sqrt{g} + 5
\quad \textrm{and} \quad
4\sqrt{h} + 1 \leq \tilde m(h) \leq \tilde c_h\sqrt{h} + 5, 
$$ 
where $c_g\leq 12$ and $\tilde c_{h}\leq 10$ are positive real constants such that the sequences
$\left(c_g\right)_{g=3}^{\infty}$ and $\left(\tilde c_h\right)_{h=3}^{\infty}$ are decreasing, $\lim_{g\rightarrow \infty}c_g = 2\sqrt{7}+2<7.3$, and $\lim_{h\rightarrow \infty}\tilde c_h = \sqrt{2}\left(\sqrt{7}+1\right)<5.2.$
\end{theorem}

\begin{proof}
The lower bounds follow from Theorem~\ref{thm:lowerBound} and the inequality $\lfloor \alpha \rfloor > \alpha -1$ which holds for any  real number $\alpha$.
To prove the upper bounds, we distinguish two cases.

\noindent \textit{i)} 
From Lemma~\ref{lem:minimalDegree} follows that if $G$ has more than $ 12\left(g-1\right)$ vertices, then it has a vertex of degree at most $6$, and hence by Lemma~\ref{lem:maxVerticesByDegree} at most  $17 + 4\sqrt{g}$ vertices. The proof is concluded by observing that $ 17 + 4\sqrt{g} > 12\left(g-1\right)$ holds for any $g\leq 2$. 
The nonorientable case is analogous.

\noindent \textit{ii)} 
We start by determining the smallest $d$ such that $d\ge 6$ and
\begin{equation}
\label{eq}
\frac{6\left(2-2g\right)}{5-d} \leq 5+2d+4\sqrt{g}
\end{equation} 
for a fixed integer $g\ge 3$.
Solving (\ref{eq}) for $d$ we get that
$$
d_{g}=\frac{5-4\sqrt{g}+\sqrt{112g+120\sqrt{g}+129}}{4}
$$ 
is minimal such $d$ and it is easy to verify that for $g\ge 3$ is indeed $d_g \ge 6$.
Therefore, by Lemma \ref{lem:upperbound} $m(g)\leq 5+2d_g+4\sqrt{g}$. 
%It is easy to see that 
Clearly,
for the sequence $(c_g)_{g=3}^{\infty}$ defined by 
$$
c_{g}=\frac{5+4\sqrt{g}+\sqrt{112g+120\sqrt{g}+129}}{2\sqrt g}
$$
$m(g)\le c_g\sqrt g +5$ for every $g\ge 3$. It can be verified by standard methods that the sequence is decreasing and has the claimed limit, which completes the proof of the orientable case.
%has the desired properties. 
The nonorientable case is analogous.
\end{proof}

In the investigation of $3$-connected equimatchable graphs embeddable in a fixed surface Kawarabayashi and Plummer \cite{KP} proved that there is no such bipartite graph embeddable with face-width at least $3$ at all.
%We conclude the paper by showing that if we relax the connectivity condition, there are infinitely-many $2$-connected bipartite equimatchable graphs with any given genus.
It is easy to see that there are arbitrarily large planar bipartite $2$-connected equimatchable graphs.

\begin{prop}
For any positive integer $k$ there is a planar $2$-connected bipartite equimatchable graph with at least $k$ vertices.
\end{prop}

\begin{proof}
Clearly, for any integer $k\ge 2$ the complete bipartite graph $K_{k,2}$ %is a planar equimatchable graph with the required properties.
has the desired properties.
\end{proof}

The following theorem  shows that there are infinitely-many $2$-connected bipartite equimatchable graphs with any given genus and face-width. 

\begin{theorem}
\label{thm:2conn-bip}
For any positive integers $n$,  $g$, and $k$ there exists a $2$-connected  bipartite  equimatchable graph $G$ 
with at least $n$ vertices, orientable genus $g$, and an embedding in $S_g$ with face-width $k$.
Similarly, for any positive integers $n$,  $h$, and $k$ there exists a $2$-connected  bipartite  equimatchable graph $\tilde G$ 
with at least $n$ vertices, nonorientable genus $h$, and an embedding in $N_h$ with face-width $k$.
% such that $G$ has  orientable genus $g$, 
%$\tilde G$ has nonorientable genus $h$ and both $G$ and $\tilde G$ have a genus embedding with face-width $k$.
\end{theorem}

\begin{proof}
We prove only the orientable case, since the nonorientable case is analogous. Take a $2$-connected graph $G'$ with at least $n$ vertices,  genus $g$, and with a genus embedding $\Pi'$ with face-width $k$, for example  any sufficiently large $2$-connected triangulation with a given genus and face-width; it is well known that such triangulations exist.
%with arbitrarily-large number of vertices and face-width exist.  
We construct the desired graph $G$ starting from $G'$ by replacing every edge $e$ of $G'$ by $l$ parallel edges $e_1, \dots, e_l$ for some fixed $l\ge 2$ and subdividing every edge $e_i$ by a new vertex $y_{e_i}$. 
Denote by $B$ the set of all vertices $y_{e_i}$ of $G$, that is, $B = \{y_{e_i}; e\in E(G'), 1\le i \le l \}$. Let $A = V(G)\setminus B$.

%Denote by $B$ the set of all vertices $y_{e_i}$ of $G$, that is, $B=\cup_{e\in E(G')} \cup_{1\le i \le l}y_{e_i}$. 
%That is, $B = \{y_{e_i}; e\in E(G'), 1\le i \le l \}$.
%There is a natural identity injective mapping $\rho$ from $V(G')$ to $V(G)$; in the rest of the proof  we slightly abuse the notation by identifying the vertices $\rho(v)$ of $G$ with their preimages $v$ in $G'$.
% In the rest of the proof we slightly abuse the language by identifying the vertices of the $G'$-minor in $G$ with the vertices of $G$.
Clearly, $G$ is bipartite and the vertices of $A$ form the smaller partite set of $G$. 
By \cite[Theorem 3]{LPP} a connected bipartite graph is equimatchable if and only if for any vertex $u$ from the smaller partite set 
%all $u \in U$
 there exists a non-empty $X\subseteq \N(u)$ such that $|\N(X)|\leq|X|$.
We prove that $G$ is equimatchable by exhibiting such set $X_v$ for every vertex $v$ from $A$.
If a vertex $v$ is in $G'$ incident with an edge $e=uv$, then 
let $X_v= \{y_{e_1}, \dots y_{e_l}\}$. In $G$ we have
%in $H$ for $X= \{y_{e_1}, \dots y_{e_k}\}$ holds 
$X_v\subseteq N(v)$ and $N(X_v)=\{u,v\}$, with possibly $u=v$ if $uv$ is a loop. 
Since $l\ge 2$, we have $|N(X_v)|\le |X_v|$.
Therefore, for every vertex $v$ from $A$ there exists a non-empty set $X_v$ such that  $X_v\subseteq N(v)$ and  $|N(X_v)|\leq |X_v|$, and hence by \cite{LPP} $H$ is equimatchable. 
%Let $\Pi$ be a genus embedding of $G$ with face width $l$. 
It is easy to see that multiplying and subdividing edges does not change the genus of the graph, and thus $\gamma(G) = \gamma(H)$.
To construct the desired genus embedding $\Pi$ of $G$ with face-width $k$, start with $\Pi'$. For any edge $e=uv$ of $G'$, choose the preferred direction of $e$.
If the preferred direction of $e$ is from $u$ to $v$, then in the rotation at $u$
% is from $u$ to $v$. In the rotation scheme 
replace the occurrence of $e$ by $e_1\ldots e_k$ and replace the occurrence of $e^{-1}$ in the rotation at $v$ by $e_k\ldots e_1$. Finally, subdivide every edge $e_i$ by the new vertex $y_{e_i}$. Clearly, the subdivided edges $e_1,\ldots, e_k$ bound $l-1$ faces of length $4$. Moreover, the occurrence of $e=uv$ in its face boundary is replaced by a sequence of two edges $(uy_{e_1})(y_{e_1}v)$ and the occurrence of $e^{-1}$ in its faces boundary is replaced by $(vy_{e_k})(y_{e_k}u)$. 
%Moreover, 
%
It is not difficult to see that  union of any $m$ faces of $\Pi$ is union of at most $m$ faces of $\Pi'$ and hence the face-width of $\Pi$ is at least $k$.
%
%the face-width of $\Pi$ is at least $k$. 
%It is known that for any embedding with face-width $k$, there is a noncontractible curve $\mathcal{C}$ intersecting the graph at $k$ points such that all points of intersection are vertices, see \cite{RV:1991}.
Since in $\Pi'$ there is a noncontractible curve of minimum length that intersects only vertices of $G'$ (see \cite{RV:1990}), there is a homotopically equivalent noncontractible curve whose intersection with $G$ consists from precisely $k$ vertices of $G$. Thus face-width of $\Pi$ is at most $k$, which completes the proof.
%
%there is always a noncontractible curve intersecting the graph 
%Since there is always a  noncontractible curve of minimum length that interscts the graph only at vertices, there is a noncontractible curve of length $k$ that intersects $G$ only at vertices of $A$ and hence the face-width of $G$ is exactly $k$. 
\end{proof}

 Theorem \ref{thm:2conn-bip} and the results of \cite{KP} suggest the following open problem.

\begin{problem}
Are there infinitely-many $3$-connected bipartite equimatchable graphs embeddable in a given surface with face-width at most $2$?
\end{problem}

%An additional question is to precisely determine the maximum size of $2$-connected equimatchable factor-critical graphs embeddable in a fixed surface.

\section*{Acknowledgement}
Research reported in this paper was partially supported by grants APVV-0223-10 and Vega 1/1005/12.

%\section*{to do}
%\begin{itemize}
%\item konstrukcia 2-suvislych podla Favaron + tesny odhad na dvojsuvisle
%\end{itemize}

\end{document}